\documentclass[12pt,fleqn]{amsart}

\usepackage[utf8]{inputenc}
\usepackage[T1]{fontenc}

\usepackage{a4wide}
\usepackage{scrextend}

\usepackage{amssymb,amsthm,latexsym}
\usepackage{mathrsfs}
\usepackage{dsfont}
\usepackage[mathscr]{euscript}
\usepackage{esint}
\usepackage{newtxtext,newtxmath} 
\usepackage[usenames,dvipsnames,x11names]{xcolor}

\usepackage{url}
\usepackage{hyperref}

\newtheorem{theorem}{Theorem}[section]
\newtheorem{lemma}[theorem]{Lemma}
\newtheorem{corollary}[theorem]{Corollary}

\newtheorem{definition}[theorem]{Definition}

\newcounter{maintheorem}

\newtheorem{mainth}[maintheorem]{Theorem}
\theoremstyle{definition}
\newtheorem*{remark*}{Remark}
\newtheorem{remark}[theorem]{Remark}

\numberwithin{equation}{section}
\theoremstyle{remark}

\newtheorem{question}{Question}

\newcommand{\vertiii}[1]{{\left\vert\kern-0.25ex\left\vert\kern-0.25ex\left\vert #1 
    \right\vert\kern-0.25ex\right\vert\kern-0.25ex\right\vert}}

\newcounter{smallromans}

\newenvironment{romanenumerate}
{\begin{list}{{\normalfont\textrm{(\roman{smallromans})}}}%
  {\usecounter{smallromans}\setlength{\itemindent}{0cm}%
   \setlength{\leftmargin}{5.5ex}\setlength{\labelwidth}{5.5ex}%
   \setlength{\topsep}{.5ex}\setlength{\partopsep}{.5ex}%
   \setlength{\itemsep}{0.1ex}}}%
{\end{list}}

\newcounter{smallromansdash}

{\end{list}}

\newcounter{bigromans}

{\end{list}}

\newcommand{\e}{\varepsilon}

\title[Differential embeddings into algebras of topological stable rank 1]{Differential embeddings into algebras of\\ topological stable rank 1}

\dedicatory{In memoriam: H.~Garth~Dales (1944--2022)}
\keywords{differential subalgebra, open multiplication, Banach algebra, ultrapower, covering dimension, norm-controlled inversion}
\author[T.~Kania]{Tomasz Kania}
\address[T.~Kania]{Mathematical Institute\\Czech Academy of Sciences\\\v Zitn\'a 25 \\115 67 Praha 1\\Czech Republic  and  Institute of Mathematics and Computer Science\\ Jagiellonian University\\ {\L}ojasiewicza 6, 30-348 Krak\'{o}w, Poland
}
\email{kania@math.cas.cz, tomasz.marcin.kania@gmail.com}

\author[N.~Ma\'slany]{Natalia Ma\'slany}
\address[N.~Ma\'slany]{Mathematical Institute\\Czech Academy of Sciences\\\v Zitn\'a 25 \\115 67 Praha 1\\Czech Republic, Department of Mathematics \\ Jan Kochanowski University in Kielce \\ \.Zeromskiego 5
\\25-369 Kielce \\Poland, and  Institute of Mathematics and Computer Science  \\ Jagiellonian University\\ {\L}ojasiewicza 6, 30-348 Krak\'{o}w, Poland
}
\email{nataliamaslany97@gmail.com}
\subjclass[2010]{46A30 (primary), and 46J10 (secondary)}

\thanks{IM CAS (RVO 67985840)}

\date{\today}

\begin{document}
\begin{abstract}
    We identify a class of \emph{smooth} Banach *-algebras that are differential subalgebras of commutative C*-algebras whose openness of multiplication is completely determined by the topological stable rank of the target C*-algebra. 
    We then show that group algebras of Abelian groups of unbounded exponent fail to have uniformly open convolution. Finally, we completely characterise in the complex case (uniform) openness of multiplication in algebras of continuous functions in terms of the covering dimension.
\end{abstract}
\maketitle
\section{Introduction}Gelfand's proof of Wiener's lemma \cite{gelfand1941normierte}, which asserts that the reciprocal of a function with absolutely convergent Fourier series that does not vanish anywhere has absolutely convergent Fourier series too, was central to the development of Banach-algebraic ramifications of Harmonic Analysis. Wiener's lemma may be rephrased as follows: the algebra of absolutely convergent Fourier series is inverse-closed when embedded into the algebra of all continuous functions on the unit circle. In the present paper, we shall be concerned with algebras that have even a stronger property, namely that the norm of an invertible element is a function of the norm of the element and its supremum norm of its Gelfand transform (see Theorem~\ref{lem:diff}); a property that the algebra of absolutely convergent series lacks.\smallskip

When $A$ is a unital Banach algebra and $i\colon A\to B$ is a unital continuous injective homomorphism, we say that $A$ \emph{admits norm-controlled inversion} in $B$, whenever there exists a~function $h\colon (0, \infty)^2 \to (0, \infty)$ so that for every element $a\in A$, which is invertible in $B,$ we have 
\[
    \|a^{-1}\|_A \leqslant h(\|a\|_A, \|i(a^{-1})\|_B).
\] 
Since the embedding $i$ is injective, an invertible element $a$ in algebra $A$ remains invertible in $B$ (strictly speaking, $i(a)$ is invertible), however in this case  the norm-controlled inversion of $A$ in $B$ implies that the inverses are actually in $A$ (\emph{i.e.}, $i(A)$ is inverse-closed in $B$).\medskip

Following Nikolskii \cite{nikolski1999search}, for $\delta > 1$, we say that a Banach algebra $A$ is $\delta$-\emph{visible in }$B$, whenever 
\begin{equation}\label{eq:psi}
    \psi\big({\delta}^{-1}\big) = \sup\{\|a^{-1}\|_A \colon a\in A, \|a\|_A \leqslant 1, \|i(a^{-1})\|_B \leqslant \delta\}<\infty.
\end{equation}
Then $A$ admits norm-controlled inversion in $B$ if and only if it is $\delta$-visible in $B$ for all $\delta>1$. Should that be the case, the norm-control function $h$ can be arranged to be
\begin{equation}\label{eq:control}
    h(\|a\|_A, \|i(a^{-1})\|_B) = \frac{1}{\|a\|_A}\psi\big( {\|a\|_A\|i(a^{-1})\|_B}\big).
\end{equation}

For a commutative (*-)semi-simple Banach (*-)algebra $A$ we say, for short, that $A$ \emph{admits norm-controlled inversion}, whenever it admits norm-controlled inversion in $C(\Phi_A)$, the space of continuous functions on the maximal (*-)ideal space $\Phi_A$ of $A$, when embedded by the Gelfand transform. (For a commutative (*-)semi-simple Banach (*-)algebra the Gelfand transform is injective; see also \cite[Proposition 30.2(ii)]{doran2018characterizations}.)\medskip

{The Wiener (convolution) algebra $\ell_1(\mathbb Z)$ is a primary example of a commutative Banach *-algebra without the norm-controlled inversion in $C(\mathbb T)$, the algebra of continuous functions on the unit circle. Indeed, in \cite{nikolski1999search} Nikolskii showed that for $\delta \geqslant 2$ we have $\psi (\delta^{-1}) = \infty,$ where $\psi$ is given in \eqref{eq:psi}.} The same conclusion extends to convolution algebras $\ell_1(G)$ for any infinite Abelian group $G$ that lack norm-controlled inversion in $C(\widehat{G})$, the algebra of continuous functions on the Pontryagin-dual group to $G$, but this behaviour appears rather exceptional. On the positive side, various weighted algebras of Fourier series (see \cite{el1998estimates}) as well as algebras of Lipschitz functions on compact subsets of Euclidean spaces enjoy the norm-controlled inversion.\smallskip

Norm-controlled inversion is a consequence of smoothness of the embedding as observed by Blackadar and Cuntz \cite{blackadar1991differential}. More specifically, let $i\colon A\to B$ be a unital injective homomorphism of unital Banach algebras. (By a subalgebra of a unital algebra we shall always mean a unital subalgebra. Likewise, all homomorphisms between unital Banach algebras are assumed to preserve the unit.) Then $A$ is a~\emph{differential subalgebra} of $B$ whenever there exists $D > 0$ such that for all $a,b\in A$ we have
\begin{equation}\label{dif}
    \|ab\|_A \leqslant D(\|a\|_A\|i(b)\|_B+\|i(a)\|_B\|b\|_A).
\end{equation}
When $A$ and $B$ are Banach *-algebras, we additionally require that $i$ is *-preserving (hence it preserves the modulus); we omit the symbol $i$, when the map $i$ is clear from the context (for example, when it is the formal inclusion of algebras). Differential subalgebras (especially of C*-algebras) have been extensively studied; see, \emph{e.g.}, \cite{kissin1994differential, grochenig2013norm, grochenig2014norm, samei2019norm}. \smallskip

A unital Banach *-algebra $A$ is \emph{symmetric}, if the spectrum of positive elements is non-negative (\emph{i.e.} $\sigma (A(a^{*}a)) \subseteq [0, \infty)$ for all $a \in A$), which means that for any $a\in A$ the element $1+a^*a$ is invertible. (See \cite[Chapter 6]{doran2018characterizations} for further characterisations.) 
In the sequel, we shall make use of \cite[Theorem 1.1(i)]{grochenig2013norm} that we record below:
\begin{theorem}\label{lem:diff}
    Differential *-subalgebras of C*-algebras have norm-controlled inversion.
\end{theorem}
In particular, differential *-algebras of C*-algebras are symmetric. We note that the condition of being a differential norm is a rather mild assumption, and norms satisfying \eqref{dif} meet a \textit{weak} form of smoothness as explained in \cite[Theorem 1.1(v)]{grochenig2013norm}).\bigskip

In the present paper, we investigate the possible connections between smoothness of an embedding of Banach algebras and topological stable rank 1 (which for unital Banach algebras is equivalent to having dense invertible group) with the openness of multiplication of a given Banach algebra $A$, \emph{i.e.}, the question of for which Banach algebras the map $m\colon A\times A\to A$ given by $m(a,b)=ab$ ($a, b\in A$) is open, that is, it maps open sets to open sets. The problem of which Banach algebras have open multiplication was systematically investigated by Draga and the first-named author in \cite{DraKa}, where it was observed that unital Banach algebras with open multiplication have topological stable rank 1 but not \emph{vice versa}. For example, matrix algebras $M_n$ have topological stable rank 1 but multiplication therein is not open unless $n=1$ (\cite{behrends2017matrix}). On the other hand, the problem of openness of convolution in $\ell_1(\mathbb Z)$ is persistently \emph{open}. 
\smallskip

Various function algebras have been observed to have open multiplication (even uniformly, where a map $f\colon X\to Y$ is uniformly open whenever for every $\varepsilon > 0$ there is $\delta > 0$ such that for all $x\in X$ one has $f(B(x,\varepsilon)) \supseteq B(f(x), \delta)$): spaces of continuous/bounded functions: \cite{balcerzak2016certain, balcerzak2011multiplication, balcerzak2005multiplying, behrends2011walk, behrends2011products, botelho2019topological, komisarski2006connection, renaud2019topological} and spaces of functions of bounded variation: \cite{canarias2021multiplication, kowalczyk2019multiplication}. The first main result of the paper unifies various approaches to openness of multiplication. (All unexplained terminology may be found in the subsequent section.)

\begin{mainth}\label{thm:A}
    Suppose that $A$ is a unital Banach *-algebra such that there exists an injective *-homomorphisim $i\colon A\to C(X)$ for some compact space $X$ such that $A$ has norm-controlled inversion in $C(X)$. Let us consider either case:
    \begin{itemize}
        \item $A = C(X)$, 
        \item $A = E^*$ is a dual Banach algebra that shares with $X$ densely many points.
    \end{itemize}
    Then multiplication in $A$ is open at all pairs of jointly non-degenerate elements. 
    
    {Furthermore, suppose that $i$ has dense range in $C(X)$. If $A$ has open multiplication, then the maximal ideal space of $A$ is of dimension at most 1.}\end{mainth}
    
Theorem~\ref{thm:A} applies, in particular, to $A = C(X)$, which may be interpreted as a complex counterpart of the main result of \cite{behrends2017pointwise}.\medskip

{Since the bidual of $C(X)$ is isometric to $C(Z)$ for some compact, zero-dimensional space (in particular, $C(X)^{**}$ has uniformly open multiplication), using Lemmas \ref{l21} and \ref{lem:perm} we may record the following corollary.}
%
%
{\begin{corollary}
  Suppose that $A$ is an Arens-regular Banach *-algebra that is densely embedded as a differential subalgebra of $C(X)$ for some compact space $X$. Then $A^{**}$ has open multiplication at all pairs of jointly non-degenerate elements.
\end{corollary}}

The proofs of the main results of \cite{canarias2021multiplication, kowalczyk2019multiplication} centre around showing that the algebras of functions of $p$-bounded variation (for $p = 1$ and $p\in (1,\infty)$, respectively) are approximable by jointly non-degenerate products. Our theorem appears to be the first general providing sufficient conditions for openness in a given commutative Banach *-algebra (\emph{i.e.}, a self-adjoint function algebra).\medskip

In \cite[Corollary 4.13]{DraKa}, Draga and the first-named author proved that $\ell_1(\mathbb Z)$ does not have uniformly open convolution (whether it is open or not remains an open problem). We strengthen this result by showing that having unbounded exponent (that is, the condition $\sup_{g\in G} o(g) = \infty$, where $o(g)$ denotes the rank of an element $g\in G)$) is sufficient for \emph{not} having uniformly open convolution. 
\begin{mainth}\label{thm:B}
    Let $G$ be an Abelian group of unbounded exponent, i.e.,  $\sup_{g\in G} o(g) = \infty$. Then convolution in $\ell_1(G)$ is not uniformly open.
\end{mainth}
By Pr\"ufer's first theorem (see \cite[p.~173]{kurosh1955theory}), every Abelian group of bounded exponent is isomorphic to a direct sum of a finite number of finite cyclic groups and a direct sum of possibly infinitely many copies of a fixed finite cyclic group, so if one seeks examples of group convolution algebras with uniform multiplication, the only candidates to be found are groups that are effectively direct sums of any number of copies of a fixed cyclic group.\medskip

Finally, we establish a complex counterpart of Komisarski's result \cite{komisarski2006connection} linking openness of multiplication in the real algebra $C(X)$ of continuous functions on a compact  space $X$ with the covering dimension of $X$. In the complex case $C(X)$ has open multiplication if and only if $X$ is zero-dimensional in which case multiplication is actually uniformly open with $\delta(\varepsilon) = \varepsilon^2 / 4$ ($\varepsilon > 0$). (See also \cite[Proposition 4.16]{DraKa} for an alternative proof using direct limits that does not depend on the scalar field; we refer to \cite{engelking1978dimension} for a modern exposition of dimension theory and standard facts thereof.)

\begin{mainth}\label{thm:c}
    Let $X$ be a compact space. Then the following conditions are equivalent for the algebra $C(X)$ of continuous complex-valued functions on $X$:
    \begin{romanenumerate}
        \item $C(X)$ has open multiplication,
        \item $C(X)$ has uniformly open multiplication,
        \item the covering dimension of $X$ is at most $1$.
    \end{romanenumerate}
    Moreover, the algebras $C(X)$ have equi-uniformly open multiplications for all compact spaces of dimension at most $1$.
\end{mainth}
A necessary condition for a unital Banach algebra to have open multiplication is topological stable rank 1, that is, having dense group of invertible elements. For a compact space $X$ of dimension at least 2, this is not the case, so $C(X)$ does not have open multiplication (\cite[Proposition 4.4]{DraKa}). The proof of Theorem~\ref{thm:c} is split into three cases.
\begin{itemize}
    \item The first one uses a reduction to spaces being topological (planar) realisations of graphs. Here we rely on certain ideas from an unpublished manuscript of Behrends for which we have permission to include them in the present note. We kindly acknowledge this crucial contribution from Professor Behrends establishing the case of $X=[0,1]$. 
    \item Then we proceed via an inverse limit argument to conclude the result for all compact metric spaces of dimension at most $1$. 
    \item Finally, we apply a result of Madre\v{s}i\'c \cite{mardevsic1960covering} to conclude the general non-metrisable case from equi-uniform openness of multiplication of $C(X)$ for all 1-dimensional compact metric spaces $X$.
\end{itemize}

\section{Preliminaries}
\subsection{Banach algebras}

Compact spaces are assumed to be Hausdorff. All Banach algebras considered in this paper are over $\mathbb C$, the field of complex scalars unless otherwise specified. We denote by $\mathbb T$ the unit circle in the complex plane.\smallskip

A Banach algebra $A$ has \emph{topological stable rank 1}, whenever invertible elements are dense in $A$ if $A$ is unital or in the unitisation of $A$ otherwise. 
Algebras whose elements have zero-dimensional spectra have topological stable rank 1 and include biduals of $C(X)$ for a compact space $X$, the algebra of functions of bounded variation, or the algebra of compact operators on a Banach space; we refer to \cite[Section 2]{DraKa} for more details.

\subsubsection{Arens regularity, dual Banach algebras}
As observed by Arens \cite{arens1951adjoint}, the biudal of a~Banach algebra may be naturally endowed with two, rather than single one, multiplications (the left and right Arens products, denoted ${\scriptstyle\square}$, $\diamond$, respectively). Even though these multiplications may be explicitly defined, the following `computation' rule is perhaps easier to comprehend: for $f,g\in A^{**}$, where $A$ is a Banach algebra, by {Goldstine's} theorem, one may choose bounded nets $(f_j)$, $(g_i)$ from A that are weak* convergent to $f$ and $g$, respectively. Then
\begin{itemize}
    \item $f\, {\scriptstyle\square}\, g = \lim_j \lim_i f_j g_i$,
    \item $f \diamond g = \lim_i \lim_j f_j g_i$
\end{itemize}
are well-defined and do not depend on the choice of the approximating nets. A Banach algebra is \emph{Arens-regular} when the two multiplications coincide. For a locally compact space $X$, the algebra $C_0(X)$ is Arens-regular, but a group $G$, the group algebra $\ell_1(G)$ (see Section~\ref{sect:semigroup}) is Arens-regular if and only if $G$ is finite (\cite{young1973irregularity}).

A \emph{dual Banach algebra} is a Banach algebra $A$ that is a dual space to some Banach space $E$ whose multiplication is separately $\sigma(A,E)$-continuous. Notable examples of dual Banach algebras include von Neumann algebras, Banach algebras that are reflexive as Banach spaces, or biduals of Arens-regular Banach algebras; see \cite[Section 5]{daws2008ultrapowers} for more details. 

Suppose that $A = E^*$ is a dual Banach algebra and let $i\colon A\to C(X)$ be an injective homomorphism for some compact space $X$. We say that $A$ \emph{shares with $X$ densely many points} whenever there exists a dense set $Q\subset X$ such that $i^*(\delta_x) \in E$ ($x\in Q$), \emph{i.e.}, the functionals $i^*(\delta_x)$ ($x\in Q$) are $\sigma(A,E)$-continuous (here $\delta_x\in C(X)^*$ is the Dirac delta evaluation functional at $x\in X$). Since for an Arens-regular Banach algebra, the bidual endowed with the unique Arens product is a dual Banach algebra, we may record the following lemma.
\begin{lemma}\label{l21}
    Let $A$ be a unital Arens-regular Banach algebra and let $i\colon A\to C(X)$ be an~injective algebra homomorphism with dense range. Then $A^{**}$ shares with the maximal ideal space of $C(X)^{**}$ densely many points.
\end{lemma}
\begin{proof}
    Since $A$ is Arens-regular, $A^{**}$ is naturally a dual Banach algebra with the unique Arens product. Since $i^{***}$ extends $i^*$, for every $x\in X$, we have $i^{***}(\delta_x) = i^*(\delta_x)\in A^*$, so that $i^*(\delta_x)$ is $\sigma(A^{**}, A^*)$-continuous. It remains to invoke the fact that $X$ can be identified with an open dense subset of the maximal ideal space of $C(X)^{**}$ via $x\mapsto (\delta_x)^{**}=\delta_{\iota(x)}$ for some point $\iota(x)$ in the maximal ideal space of $C(X)^{**}$ (see the discussion after \cite[Definition 3.3]{dales2012second}); the map $\iota$ is necessarily discontinuous unless $X$ is finite). 
\end{proof}

Let us record two permanence properties of differential embeddings; even though we shall not utilise \eqref{lem:perm:ultra} in the present paper, we keep it for possible future reference.
\begin{lemma}\label{lem:perm}
    Let $A$ be a Banach algebra continuously embedded into another Banach algebra $B$ by a homomorphism $i\colon A\to B$ as a differential subalgebra. 
    \begin{romanenumerate}
        \item Consider both in $A^{**}$ and $B^{**}$ either left or right Arens products. Then in either setting $i^{**}\colon A^{**}\to B^{**}$ is a differential embedding.
        \item\label{lem:perm:ultra} Let $\mathscr{U}$ be an ultrafilter. Then $i^{\mathscr{U}}\colon A^{\mathscr{U}}\to B^{\mathscr{U}}$ is a differential embedding between the respective ultrapowers.
    \end{romanenumerate}
\end{lemma}
\begin{proof}\emph{Case 1.} 
    Let $\{a_{\alpha}\}, \{b_{\beta}\} \subset A$ be bounded nets $\sigma(A^{**}, A^*)$-convergent to $a,b \in A^{**}$ respectively, satisfying for any $\alpha, \beta$ conditions $\|a_{\alpha}\|_A \leqslant \|a\|_{A^{**}}$ and $\|b_{\beta}\|_B \leqslant \|b\|_{B^{**}}$ (it is possible by the Goldstine and Krein--\v{S}mulyan theorems). Then
    \begin{equation*}
        \begin{split}
            \|ab\|_{A^{**}} &\leqslant \liminf\limits_{\alpha, \beta} \|a_{\alpha} b_{\beta}\| \\
            &\leqslant D \cdot \liminf\limits_{\alpha, \beta} \, \big(\|a_{\alpha}\|_A\|i(b_{\beta})\|_B+\|i(a_{\alpha})\|_B\|b_{\beta}\|_A\big) \\
            &\leqslant D \big(\|a\|_{A^{**}}\|i^{**}(b)\|_{B^{**}}+\|i^{**}(a)\|_{B^{**}}\|b\|_{A^{**}}\big).
        \end{split}
    \end{equation*}
 \emph{Case 2.} 
    Let $a=[(a_{\gamma})_{\gamma \in \Gamma}]$, $b=[(b_{\gamma})_{\gamma \in \Gamma}] \in A^{\mathscr{U}}.$ Then
    \begin{equation*}
        \begin{split}
            \|a b \|_{A^{{\mathscr{U}}}} &= \lim_{\gamma, \mathscr{U}}\big\|a_{\gamma} b_{\gamma} \big\|_A \\
            &\leqslant \lim_{\gamma, \mathscr{U}} D \Big(\big\|a_{\gamma}\big\|_A\big\|i(b_{\gamma})\big\|_B+\big\|i(a_{\gamma})\big\|_B\big\|b_{\gamma}\big\|_A\Big) \\
            &\leqslant  D \Big( \lim_{\gamma, \mathscr{U}} \big\|a_{\gamma}\big\|_A \cdot \lim_{\gamma, \mathscr{U}}\|i(b_{\gamma})\big\|_B+\lim_{\gamma, \mathscr{U}}\big\|i(a_{\gamma})\big\|_B\cdot \lim_{\gamma, \mathscr{U}}\big\|b_{\gamma}\big\|_A\Big) \\
            &= D \Big(\big\|a\big\|_{A^{{\mathscr{U}}}}\big\|i^{{\mathscr{U}}}\big(b\big)\big\|_{B^{{\mathscr{U}}}}+\big\|i^{{\mathscr{{\mathscr{U}}}}}\big(a\big)\big\|_{B^{{\mathscr{U}}}}\big\|b\big\|_{A^{{\mathscr{U}}}}\Big).
        \end{split}
    \end{equation*} 
\end{proof}

\subsection{Banach *-algebras} Let $A$ be a unital Banach *-algebra. In this setting, for $a\in A$ we interpret $|a|^2$ as $a^*a$. We say that elements $a,b$ in $A$ are \emph{jointly non-degenerate}, when $|a|^2+|b|^2$ is invertible. When $X$ is a compact space and $a,b \in C(X)$, we sometimes say that elements with $|a|^2+|b|^2 \geqslant \eta$ (for some $\eta > 0$) are \emph{jointly $\eta$-non-degenerate}. Let us introduce the following definition.
\begin{definition}\label{df23}
    A unital Banach *-algebra $A$ is \emph{approximable by jointly non-degenerate products} whenever for all $a,b\in A$ and $\varepsilon > 0$ there exist jointly non-degenerate elements $a^\prime, b^\prime\in A$ with $\max \{ \|a-a^\prime\|, \|b-b^\prime\| \}<\varepsilon$ such that $ab = a^\prime b^\prime$.  
\end{definition}
\begin{remark}It is readily seen that $C(X)$ for a zero-dimensional compact space $X$ has this property. Indeed, let $f,g\in C(X)$ and $\varepsilon> 0$. Consider the sets
\begin{itemize}
    \item $D_1 = \{x\in X\colon |f(x)|\geqslant \varepsilon/ 3\}$
    \item $D_2 =  \{x\in X\colon |g(x)|\geqslant \varepsilon/ 3\}$
    \item $D_3 = \{x\in X\colon |f(x)|, |g(x)| \leqslant \varepsilon/2\}$.
\end{itemize}
Certainly, the sets $D_1, D_2, D_3$ are closed and cover the space $X$. As $X$ is zero-dimensional, there exist pairwise clopen sets $D_1^\prime \subseteq D_1$, $D_2^\prime \subseteq D_2$, and $D_3^\prime \subseteq D_3$ that still cover $X$, \emph{i.e.}, $X = D_1^\prime \cup D_2^\prime \cup D_3^\prime$. Let $f^\prime = f \cdot \mathds{1}_{D_1^\prime\cup D_2^\prime} + \tfrac{\varepsilon}{2}\mathds{1}_{D_3^\prime}$ and $g^\prime = g \cdot \mathds{1}_{D_1^\prime\cup D_2^\prime} + \tfrac{2}{\varepsilon}fg\mathds{1}_{D_3^\prime}$. Then $f^\prime$, $g^\prime$ are the sought jointly non-degenerate approximants. On the other hand, as $C(X)$ for $X = [0,1]$ and compact spaces alike are readily not approximable by jointly non-degenerate issues due to connectedness.
\end{remark}

Kowalczyk and Turowska \cite{kowalczyk2019multiplication} showed that the algebra $BV[0,1]$ of functions of bounded variation on the unit interval is approximable by jointly non-degenerate products and Canarias, Karlovich, and Shargorodsky \cite{canarias2021multiplication} extended this result to algebras of bounded $p$-variation on the interval as well as certain further function algebras.

\subsection{Ultraproducts} Ultraproducts of mathematical structures usually come in two main guises: the algebraic one (first-order) and the analytic one (second-order). Let us briefly summarise the link between these in the context of groups and their group algebras. This has been essentially developed by Daws in \cite[Section 5.4]{daws2009amenability} and further explained in \cite[Section 2.3.2]{DraKa}.\smallskip

Let $(S_\gamma)_{\gamma\in \Gamma}$ be an infinite collection of semigroups and let $\mathscr U$ be an ultrafilter on $\Gamma$. The (algebraic) \emph{ultraproduct} $\prod_{\gamma\in \Gamma}^{\mathscr{U}}S_\gamma$ with respect to $\mathscr{U}$ (denoted $S^{\mathscr U}$ when $S_\gamma=S$ for all $\gamma\in \Gamma$ and then termed the \emph{ultrapower} of $S$ with respect to $\mathscr{U}$) is the quotient of the direct product $\prod_{\gamma\in \Gamma}S_\gamma$ by the congruence 
\[
    (g_\gamma)_{\gamma\in \Gamma} \sim (h_\gamma)_{\gamma\in \Gamma}\quad\text{ if and only if }\quad\{\gamma\in \Gamma\colon g_\gamma = h_\gamma\}\in \mathscr{U}.
\]
Then the just-defined ultraproduct is naturally a semigroup/group/Abelian group if $S_\gamma$ are semigroups/groups/Abelian groups for $\gamma\in \Gamma$.\smallskip

Let $(A_\gamma)_{\gamma\in \Gamma}$ be an infinite collection of Banach spaces. Then the $\ell_\infty(\Gamma)$-direct sum $A = ( \bigoplus_{\gamma\in \Gamma} A_\gamma)_{\ell_\infty(\Gamma)}$, that is, the space of all tuples $(x_\gamma)_{\gamma\in \Gamma}$ with $x_\gamma\in A_\gamma$ ($\gamma\in \Gamma$) and $\sup_{\gamma\in \Gamma}\|x_\gamma\|<\infty$ is a Banach space under the supremum norm. Moreover, the subspace $J=c_0^{\mathscr U}(A_\gamma)_{\gamma\in \Gamma}$ comprising all tuples $(x_\gamma)_{\gamma\in \Gamma}$ such that $\lim_{\gamma\to \mathscr{U}}\|x_\gamma\|=0$ is closed. The (Banach-space) \emph{ultraproduct} $\prod_{\gamma\in \Gamma}^{\mathscr U}A_\gamma$ of $(A_\gamma)_{\gamma\in \Gamma}$ with respect to $\mathscr{U}$ is the quotient space $A / J$. If $A_\gamma$ ($\gamma\in \Gamma$) are Banach algebras, then naturally so is $A$ and $J$ is then a closed ideal therein. Consequently, the ultraproduct is a Banach algebra. Let us record formally a link between these two constructions.

\begin{lemma}
    Let $(S_\gamma)_{\gamma\in \Gamma}$ be an infinite collection of semigroups and let $\mathscr{U}$ be a countably incomplete ultrafilter on $\Gamma$. Then there exists a unique contractive homomorphism
    \begin{equation}\label{homomorphismG}
        \iota \colon {\prod_{\gamma\in \Gamma}}^{\mathscr{U}}\ell_1(S_\gamma)\to \ell_1\Big({\prod_{\gamma\in \Gamma}}^{\mathscr{U}}S_\gamma \Big) 
    \end{equation}
    that satisfies 
    \[
        \iota\Big(\big[(e_{g_\gamma})_{\gamma\in\Gamma}\big]\Big) = e_{[(g_\gamma)_{\gamma\in\Gamma}]}\qquad \Big(\big[(e_{g_\gamma})_{\gamma\in\Gamma}\big]\in {\prod_{\gamma\in \Gamma}}^{\mathscr U}\ell_1(S_\gamma)\Big).
    \]
\end{lemma}

\subsection{Abelian groups} Let $G$ be a group. For $g\in G$ we denote by $o(g)$ the order of the element $g$. For a (locally compact) Abelian group $G$ we denote by $\widehat{G}$ the Pontryagin dual group of $G$; for details and basic properties concerning this duality we refer to \cite[Chapter 6]{hewitt2012abstract}.
 
If $G$ is an (Abelian) divisible group, that is, for any $g\in G$ and $n\in \mathbb N$ there is $h\in G$ such that $g = nh$, then $G$ is an injective object in the category Abelian groups, which means that for any Abelian groups $H_1 \subset H_2$, every homomorphism $\varphi\colon H_1\to G$ extends to a homomorphism $\overline{\varphi}\colon H_2\to G$. Direct sums of arbitrary many copies of $\mathbb Q$, the additive group of rationals, are divisible. \smallskip

Let us record for the future reference the following observation, likely well known to alge\-bra\-ically-orien\-ted model theorists.

\begin{lemma}\label{lem:idempotent}
    Suppose that $G$ is an Abelian group with $\sup_{g\in G} o(g) = \infty$. Then $\mathbb Z^{(\mathbb R)}$ embeds into an ultrapower of $G$ with respect to an ultrafilter on a countable set.
\end{lemma}

\begin{proof}
Let $\mathscr{U}$ be a non-principal ultrafilter on $\mathbb N$ and let $(g_n)_{n=1}^\infty$ be a sequence in $G$ such that $\sup_{n}o(g_n) = \infty$. Then $g = [(g_n)_{n=1}^\infty ]$ has infinite order in $H = G^{\mathscr{U}}$. Let $\mathscr{A}$ be an~almost disjoint family of infinite subsets of $\mathbb N$ that has cardinality continuum. Then all but at most one elements $\mathscr{A}$ are not in $\mathscr{U}$ (as $\mathscr{U}$ is non-principal and closed under finite intersections), so let us assume that $\mathscr{A} \subset \mathscr{U}^\prime$. For each $A\in \mathscr{A}$ we set
\[
    g_A(i) = 
                \left\{\begin{array}{ll}
                    g, & i \notin A, \\ 
                    0, & i \in A
                         \end{array}
                            \quad (i\in \mathbb N). \right.
\]
Then $h_A = [(g_A(i))_{i=1}^\infty] \in H^{\mathscr{U}}$ and $o(h_A) = \infty$ ($A\in \mathscr{A}$). Moreover, $\{h_A\colon A\in \mathscr A\}$ is a~$\mathbb Z$-linearly independent set of cardinality continuum. As such, the subgroup it generates is isomorphic to $\mathbb Z^{(\mathbb R)}$. It remains to notice that canonically $(G^{\mathscr{U}})^{\mathscr U} \cong G^{\mathscr{U} \otimes \mathscr{U}}$, as required.
\end{proof}

\subsection{Semigroup algebras}\label{sect:semigroup}
Let $S$ be a semigroup written multiplicatively. In the Banach space $\ell_1(S)$ one can define a convolution product by
\[
    x\ast y =\sum_{t\in S} \Big(\sum_{r\cdot s=t}x_r y_s\Big)e_t \quad (x = (x_s)_{s\in S},\; y = (y_s)_{s\in S}\in \ell_1(S)),
\]
where $(e_s)_{s\in S}$ is the canonical unit vector basis of $\ell_1(S)$, together with $\ell_1(S)$ becomes a~Banach algebra. For the additive semigroup of natural numbers, the convolution in $\ell_1(\mathbb N)$ renders the familiar Cauchy product.

Suppose that $T\subseteq S$ is a subsemigroup. Then $\ell_1(T)$ is naturally a closed subspace of $\ell_1(S)$, which is moreover a closed subalgebra. Every surjective semigroup homomorphism $\vartheta\colon T\to S$ implements a surjective homomorphism $\iota_\vartheta\colon \ell_1(T)\to \ell_1(S)$ on the Banach-algebra level by the action 
\begin{equation}\label{homomorphismtheta}
    \iota_\vartheta e_t = e_{\vartheta(t)}\quad (t\in T).
\end{equation} 
When $G$ is an Abelian (discrete) group, then the (compact) dual group $\widehat{G}$ is the maximal ideal space of the convolution algebra $\ell_1(G)$. More information on semigroup algebras may be found in \cite[Chapter 4]{dales2012second}. 

\section{Proofs of Theorems A and B}
The proof of Theorem~\ref{thm:A} relies on constructing certain approximation scheme for pairs of elements, an idea that shares similarities with the methods used in the main results of \cite{kowalczyk2019multiplication} and \cite{canarias2021multiplication}. However, the results achieved in Theorem~\ref{thm:A} are more general, and the techniques applied in its proof are correspondingly broader. Let us now proceed to the proof.

\begin{proof}[Proof of Theorem~\ref{thm:A}]
Suppose that $A$ has norm-controlled inversion implemented by a *-ho\-mo\-mor\-phism $i\colon A\to C(X)$. Since *-homomorphism of Banach *-algebra into a C*-algebra is always norm decreasing, we have
\begin{equation}\label{eq:local-0}
    \|i(f)\|_\infty \leqslant \|f\|_A\quad (f\in A).
\end{equation}
Suppose that $F,G\in A$ are jointly non-degenerate (in particular, $|F| + |G|$ is invertible in $C(X)$ being nowhere zero). Fix $\e\in(0,1)$ and let
\begin{equation}\label{eq:local-1}
    \gamma:=\min\left\{1,\frac{1}{2}\inf_{x\in X}
    \big((|i(F))(x)|+|(i(G))(x)|\big)\right\}
\end{equation}
Set
\begin{equation}\label{eq:local-2}
    K:=2\cdot\max\big\{\|F\|_A,\|G\|_A,1\big\},
\end{equation}
\begin{equation}
    \widehat{T}:= \frac{2C}{{\gamma}^2} \cdot \psi\bigg( \frac{4K^2}{{\gamma}^2}\bigg) > 0,
\end{equation}
where the function $\psi$ satisfies \eqref{eq:psi}. Moreover, let $  T:= \max \{\widehat{T}, 1\}.$ Pick an arbitrary element $H\in A$ so that
\begin{equation}\label{eq:local-3}
    \|H\|_{A}<\frac{\e \cdot \gamma}{C K^3 T^2}
\end{equation}
and consider
\begin{equation}\label{eq:local-4}
F_0:=F,\quad G_0:=G,\quad H_0:=H
\end{equation}
We then define recursively the sequences $(F_n)_{n=0}^\infty$, $(G_n)_{n=0}^\infty$, and
$(H_n)_{n=0}^\infty$ by
\begin{equation}\label{eq:local-5}
    F_{n+1} := F_n+\frac{H_n\overline{G_n}}{|F_n|^2+|G_n|^2}, G_{n+1} := G_n+\frac{H_n\overline{F_n}}{|F_n|^2+|G_n|^2}, H_{n+1} :=-\frac{H_n^2\overline{F_n G_n}}{(|F_n|^2+|G_n|^2)^2}.
\end{equation}

We \emph{claim} that
\begin{romanenumerate}
    \item \label{eq:FGHi}
    \[
        F_nG_n+H_n=FG+H\quad (n = 0, 1, 2, \ldots),
    \]
    \item \label{eq:FGHii}
    \[
        \|F_n\|_{A}, \|G_n\|_{A}\leqslant \tfrac{1}{2}K+1-2^{-n}<K,
    \]
    \item \label{eq:FGHiii}
    \[
        \inf_{x\in X}\big(|(i(F_n))(x)|+|(i(G_n))(x)|\big)\geqslant \gamma +\gamma\cdot 2^{-n}>0,
    \]
    \item \label{eq:FGHiv}
    \[
        \|H_n\|_{A}\leqslant \frac{1}{2^n}\cdot\frac{\e \cdot \gamma}{C K^3 T^2}.
    \]
\end{romanenumerate}
Note that \eqref{eq:FGHiii} implies that sequences \eqref{eq:local-5} are well defined. We will prove these claims by induction. 

It follows from 
\eqref{eq:local-4} that $F_0G_0+H_0=FG+H.$ We obtain from \eqref{eq:local-1}--\eqref{eq:local-4} that
\begin{itemize}
    \item $ \|F_0\|_{A}=\|F\|_{A}\leqslant K/2$,
    \item $\|G_0\|_{A}=\|G\|_{A}\leqslant K/2$, 
    \item $\|H_0\|_{A}=\|H\|_{A}<\frac{\e \cdot \gamma}{CK^3T^2}$,
    \item $\inf_{x\in X}\big(|F_0(x)|+|G_0(x)|\big) =
            \inf_{x\in X}\big(|F(x)|+|G(x)|\big)\geqslant 2\gamma >0.$
\end{itemize}
That is, \eqref{eq:FGHi}--\eqref{eq:FGHiv} are satisfied for $n=0$.

Now we assume that \eqref{eq:FGHi}--\eqref{eq:FGHiv} are fulfilled for some $n = 0, 1,2, \ldots $ Consequently, sequences \eqref{eq:local-5} are well defined. 
Then, taking into account \eqref{eq:local-2}, we see that $K/2\geqslant 1$ and
\begin{equation}\label{eq:local-8}
F_nG_n+H_n=FG+H,
\end{equation}
\begin{equation}\label{eq:local-9}
\|F_n\|_{A}\leqslant \frac{K}{2}+1-2^{-n}<K,
\end{equation}
\begin{equation}\label{eq:local-10}
\|G_n\|_{A}\leqslant \frac{K}{2}+1-2^{-n}<K,
\end{equation}
\begin{equation}\label{eq:local-11}
\inf_{x\in X}\big(|(i(F_n))(x)|+|(i(G_n))(x)|\big)\geqslant \gamma+\gamma\cdot 2^{-n}>\gamma,
\end{equation}
\begin{equation}\label{eq:local-12}
\|H_n\|_{A}\leqslant \e\cdot 2^{-n}\cdot\frac{\gamma}{CK^3T^2}.
\end{equation}
Let us show that \eqref{eq:FGHi}--\eqref{eq:FGHiv} are fulfilled for $n+1$.

For \eqref{eq:FGHi}, it follows from \eqref{eq:local-5}--\eqref{eq:local-8} that
\[
    \begin{split}
    F_{n+1}G_{n+1}+H_{n+1}
    &=
    \left(F_n+\frac{H_n\cdot\overline{G_n}}{|F_n|^2+|G_n|^2}\right)
    \left(G_n+\frac{H_n\cdot\overline{F_n}}{|F_n|^2+|G_n|^2}\right)
    -
    \frac{H_n^2\cdot\overline{F_nG_n}}{(|F_n|^2+|G_n|^2)^2}
    \\
    &= 
    F_nG_n+H_n\frac{F_n\overline{F_n}+G_n\overline{G_n}}{|F_n|^2+|G_n|^2}
    +
    H_n^2\frac{\overline{F_nG_n}}{(|F_n|^2+|G_n|^2)^2}
    -
    H_n^2\frac{\overline{F_nG_n}}{(|F_n|^2+|G_n|^2)^2}
    \\
    &= 
    F_nG_n+H_n=FG+H.
    \end{split}
\]
Hence, \eqref{eq:FGHi} is satisfied for $n+1$.

As for \eqref{eq:FGHii}, using \eqref{eq:local-9}--\eqref{eq:local-10} we conclude that
\begin{equation}\label{k}
    \begin{array}{lcl}
        \| |F_n|^2+|G_n|^2\|_{A}
        &\leqslant &
        \|F_n\cdot \overline{F_n}\|_{A}
        +
        \|G_n\cdot\overline{G_n}\|_{A} \\
        &\leqslant &
        \| F_n\|_{A}
        \|\overline{F_n}\|_{A}
        +
        \|G_n\|_{A}\|\overline{G_n}\|_{A}
        \\
        &= &
        \|F_n\|_{A}^2 +\|G_n\|_{A}^2\\
        &\leqslant &
        2K^2.
    \end{array}
\end{equation}\label{eq:local-13}
It follows from \eqref{eq:local-11} that
\[
    \begin{array}{lcl}
        \gamma^2 &
        \leqslant &
        \inf\limits_{x\in X}\big(|(i(F_n))(x)|+|(i(G_n))(x)|\big)^2\\
        & = & \inf\limits_{x\in X} ( |(i(F_n))(x)|^2+2|(i(F_n))(x)|\cdot|(i(G_n))(x)|+|(i(G_n))(x)|^2) \\
        & \leqslant &
        2\inf\limits_{x\in X}\big(|(i(F_n))(x)|^2+|(i(G_n))(x)|^2\big),
    \end{array}
\]
hence
\begin{equation}\label{eq:local-14}
    \sup_{x\in X}\big(|(i(F_n))(x)|^2+|(i(G_n))(x)|^2\big)\geqslant  \inf_{x\in X}\big(|(i(F_n))(x)|^2+|(i(G_n))(x)|^2\big)\geqslant \frac{\gamma^2}{2}>0.
\end{equation}
By \eqref{eq:local-0} and \eqref{eq:local-14}  we obtain
\begin{equation}\label{eq:local-144}
    \| |F_n|^2+|G_n|^2\|_{A} \geqslant \frac{1}{C} \cdot \frac{\gamma^2}{2}>0.
\end{equation}
It then follows from \eqref{eq:local-5}, \eqref{eq:local-9}--\eqref{eq:local-10} and  \eqref{eq:local-14} that
\begin{equation}\label{eq:local-15}
    \begin{split}
        \|F_{n+1}\|_{A}
        &\leqslant 
        \|F_n\|_{A}+\|H_n\|_{A}\|G_n\|_{A}
        \left\|\frac{1}{|F_n|^2+|G_n|^2}\right\|_{A}\\
        &\leqslant 
        \left(\frac{K}{2}+1-2^{-n}\right)+\|H_n\|_{A}K
        \left\|\frac{1}{|F_n|^2+|G_n|^2}\right\|_{A}.
    \end{split}
\end{equation}
Since $A$ admits norm-controlled inversion in $C(X)$, we follows from \eqref{k}, \eqref{eq:local-14}, \eqref{eq:local-144}  that
\begin{equation}\label{eq:local-16}
    \begin{split}
        \left\|\frac{1}{|F_n|^2+|G_n|^2}\right\|_A 
        &\leqslant \frac{1}{\||F_n|^2+|G_n|^2\|_A}
        \cdot \psi\Big({\||F_n|^2+|G_n|^2\|_A
        \cdot \big\|\big(|F_n|^2+|G_n|^2\big)^{-1}\big\|_{\infty}}\Big) \\
        &\leqslant \frac{2C}{{\gamma}^2}
        \cdot \psi\bigg({2K^2
        \cdot \frac{2}{\gamma^2}}\bigg)= \widehat{T}.
    \end{split}
\end{equation}
Combining \eqref{eq:local-15}--\eqref{eq:local-16} with \eqref{eq:local-12}
and taking into account that $\e \in(0,1),$  $\gamma \in(0,1],$ $ K \geqslant 2,$ $C\geqslant 1$, and $T\geqslant 1$, we obtain
\begin{equation}\label{eq:local-17}
    \begin{array}{lcl}
        \|F_{n+1}\|_{A}, \|G_{n+1}\|_{A}
        &\leqslant &
        \frac{K}{2}+1-2^{-n}+ K \widehat{T} \cdot\e\cdot 2^{-n}\cdot\frac{\gamma}{CK^3T^2}\\
        &\leq&
        \frac{K}{2}+1-2^{-n}+2^{-n} \cdot \frac{1}{2}\\
        & = & \frac{K}{2}+1-2^{-n-1}.
    \end{array}
\end{equation}
Thus, \eqref{eq:FGHii} is fulfilled for $n+1$.

In order to verify \eqref{eq:FGHiii}, since $\e \in(0,1),$  $\gamma \in(0,1],$ $K \geqslant 2,$ $C\geqslant 1$, and $T\geqslant 1$,
it follows from \eqref{eq:local-5}, \eqref{eq:local-0},
\eqref{eq:local-10}, \eqref{eq:local-12} and \eqref{eq:local-16} that for $x\in X$ we have
\begin{equation*}
    \begin{split}
        |(i(F_n))(x)|
        &\leq
        |(i(F_{n+1}))(x)|+|(i(H_n))(x)|\frac{|(i(G_n))(x)|}{|(i(F_n))(x)|^2+|(i(G_n))(x)|^2}\\
        &\leqslant 
        |(i(F_{n+1}))(x)|+ C\|H_n\|_A\|G_n\|_A
        \left\|\frac{1}{|F_n|^2+|G_n|^2}\right\|_A
        \\
        &\leqslant 
        |(i(F_{n+1}))(x)|+
        C \cdot \e\cdot 2^{-n}\frac{\gamma}{CK^3T^2}\cdot K\widehat{T}\\
        &<
        |(i(F_{n+1}))(x)|+2^{-n}\cdot\frac{\gamma}{K^2}\\
        &\leqslant
        |(i(F_{n+1})(x)|+2^{-n}\cdot\frac{\gamma}{4}.
    \end{split}
\end{equation*}
Consequently,
\begin{equation}\label{eq:local-19}
    |(i(F_{n+1}))(x)|>|(i(F_n))(x)|-2^{-n-2}\gamma \quad (x\in X).
\end{equation}
In the same way, we observe that
\begin{equation}\label{eq:local-20}
    |(i(G_{n+1}))(x)|>|(i(G_n))(x)|-2^{-n-2}\gamma \quad (x\in X).
\end{equation}
We conclude from \eqref{eq:local-11} and 
\eqref{eq:local-19}--\eqref{eq:local-20} that 
\begin{equation*}
    \begin{split}
        \inf_{x\in X}\big(|(i(F_{n+1}))(x)|+|(i(G_{n+1}))(x)|\big)
        &\geqslant 
        \inf_{x\in X}\big(|(i(F_{n}))(x)|+|(i(G_{n}))(x)|\big)
        -2\cdot 2^{-n-2}\gamma 
        \\
        &\geqslant 
        \gamma+\gamma\cdot 2^{-n}-\gamma\cdot 2^{-n-1}\\
        & =\gamma+\gamma\cdot 2^{-n-1},
    \end{split}
    \end{equation*}
so \eqref{eq:FGHiii} is fulfilled for $n+1$.

Finally, for \eqref{eq:FGHiv}, by \eqref{eq:local-9}--\eqref{eq:local-10}, \eqref{eq:local-12} and \eqref{eq:local-16}, for $\e \in(0,1),$  $\gamma \in(0,1],$ $ K \geqslant 2,$ and $C\geqslant 1$, we then have
\begin{equation*}
    \begin{array}{lcl}
        \|H_{n+1}\|_{A}
        &\leqslant & 
        \|H_n\|_{A}^2
        \|\overline{F_n}\|_{A}
        \|\overline{G_n}\|_{A}
        \left\|\frac{1}{|F_n|^2+|G_n|^2}\right\|_{A}^2\\
        &=& 
        \|H_n\|_{A}^2
        \|F_n\|_{A}
        \|G_n\|_{A}
        \left\|\frac{1}{|F_n|^2+|G_n|^2}\right\|_{A}^2
        \\
        &\leqslant &
        \left(\e\cdot 2^{-n}\cdot\frac{\gamma}{CK^3T^2}\right)^2
        \cdot K^2 \cdot \widehat{T}^2\\
        &\leqslant &
        \e \cdot 2^{-n}\cdot\frac{\gamma}{CK^3T^2} \cdot\frac{\gamma}{CK^3T^2} \cdot K^2 \cdot \widehat{T}^2\\
        &\leqslant &
        \e \cdot 2^{-n}\cdot\frac{\gamma}{CK^3T^2} \cdot \frac{1}{K}\\
        & \leqslant &
        \e\cdot 2^{-n-1}\cdot\frac{\gamma}{C K^3 T^2},
    \end{array}
\end{equation*}
which verifies \eqref{eq:FGHiv} for $n+1$. 

It follows from \eqref{eq:local-0} and \eqref{eq:FGHiv} that 
\begin{equation}\label{eq:local-22}
    \lim_{n\to\infty}|(i(H_n))(x)|  \leqslant  C\lim_{n\to\infty}\|H_n\|_A
    \leqslant 
    \e\cdot\frac{\gamma}{K^3T^2}\lim_{n\to\infty}2^{-n}=0 \quad (x\in X).
\end{equation}

Suppose that $m,n \in \mathbb{N}$, $m>n.$ For
$\e \in(0,1),$  $\gamma \in(0,1],$ $ K \geqslant 2,$ $C\geqslant 1$, and $T\geqslant 1$,
by 
\eqref{eq:local-5}, \eqref{eq:local-10}, \eqref{eq:local-12},
\eqref{eq:local-16}, we observe that
\begin{equation}\label{eq:local-244}
    \begin{split} 
        \sum_{n=0}^{\infty} \|F_{n+1}-F_n\|_{A}&\leqslant  
        \sum_{n=0}^\infty 
        \|H_n\|_{A}
        \|G_n\|_{A}
        \left\|\frac{1}{|F_n|^2+|G_n|^2}\right\|_{A}\\
        &\leqslant 
        \sum_{n=0}^\infty \frac{1}{2^n}\cdot\frac{\e\gamma}{CK^3T^2}
        \cdot K\widehat{T} \\
        & \leqslant 
        \e \cdot \frac{1}{K^2}\sum_{n=0}^\infty 2^{-n}\\
        &<
        \e \cdot \frac{1}{2}\cdot \sum_{n=0}^\infty \frac{1}{2^n} < \e.
    \end{split}
\end{equation}
\begin{itemize}
\item[] \emph{Case 1.}
From \eqref{eq:local-244} for any $\varepsilon_{1}>0$ there exist $N$ such that for $m,n \in \mathbb{N}, m>n>N$ holds
\begin{equation}\label{eq:local-2444}
    \begin{split}
        \|F_m - F_n\|_{\infty}
        &\leqslant  
        \sum_{j=n}^{m-1} \|F_{j+1}-F_j\|_{\infty}< \varepsilon_1,
    \end{split}
\end{equation}
which means that the sequence $(F_n)_{n=1}^\infty$ is uniformly Cauchy, so it converges uniformly to some continuous function $f$. Similarly, there exists a continuous function $g$ that is the limit of the uniformly convergent sequence $(G_n)_{n=1}^{\infty}$.\smallskip

In particular we obtain
\begin{equation}\label{point}
    \lim_{n\rightarrow \infty} F_n(x) = f(x) \quad \text{ and } \quad \lim_{n\rightarrow \infty} G_n(x) = g(x).
\end{equation}

Using \eqref{point}, \eqref{eq:FGHi} and \eqref{eq:local-22},
we see that
\begin{equation}\label{eq:local-223}
    \begin{array}{lcl}
    f(x)\cdot g(x) & = & \lim\limits_{n\to\infty} \big( F_{n}(x) \cdot G_{n}(x) \big)\\
        & = & \lim\limits_{n\to\infty}\big( F_{n}(x) \cdot G_{n}(x)+H_{n}(x)\big) \\
        & = & F(x) \cdot G(x)+H(x).
    \end{array}
\end{equation}
Moreover, from \eqref{eq:local-244} we have 
\begin{equation}\label{in}
    \begin{split}
        \|f-F\|_{\infty}
        &\leqslant 
        \sum_{n=0}^\infty \|F_{n+1}-F_n\|_{\infty}
         < \e.
    \end{split}
\end{equation}
We show that $\|g-G\|_{\infty}<\e$ in the same way.
\item[] \emph{Case 2.} $A$ is a dual Banach algebra with $A = E^*$ that shares with $X$ densely many points as witnessed by some dense set $Q\subset X$.

In view of \eqref{eq:FGHii}, the sequences $(F_n)_{n=0}^\infty$ and $(G_n)_{n=0}^\infty$ are uniformly bounded by constant $K$. Let $\mathscr U$ be a non-principal ultrafilter on $\mathbb N$. By the Banach--Alaoglu theorem, $(F_n)_{n=0}^\infty$ and $(G_n)_{n=0}^\infty$ converge to some elements $f,g\in A$, $\|f\|, \|g\|\leqslant K$ with respect to $\sigma(A,E)$ along $\mathscr{U}$.
%
Using \eqref{eq:FGHi} and \eqref{eq:local-22}, we see that for any $x\in Q$
\begin{equation}\label{eq:local-23}
    \begin{split}
        (i(f))(x)\cdot (i(g))(x) & =  \langle \delta_x, i(fg)\rangle\\
        & =  \langle i^*(\delta_x), fg\rangle\\
        & = \lim_{n\to\mathscr{U}}\langle i^*(\delta_x), F_{n}G_n\rangle\\
          & = \lim_{n\to\mathscr{U}}\langle \delta_x, i(F_{n}G_n)\rangle\\
        & =  \lim_{n\to\mathscr{U}} \big( (i(F_{n}))(x)\cdot (i(G_{n}))(x) \big)\\
        & = \lim_{n\to\mathscr{U}}\big((i(F_{n}))(x) \cdot (i(G_{n}))(x)+(i(H_{n}))(x)\big) 
    \end{split}
\end{equation}
nonetheless, it follows from \eqref{eq:FGHi} that 
\[
    \| i\big(F_n G_n + H_n - (FG+H)\big)\|_{\infty}  \leqslant C \cdot \| F_n G_n + H_n - (FG+H)\|_{A} = 0,
\]
hence for any $x \in Q$ we have
\begin{equation}\label{eq:local-2333}
    \begin{split}
        i \big( f(x)g(x) \big) = i \big( F(x)G(x)+H(x) \big).
    \end{split}
\end{equation}
Since $Q$ is a dense subset of $X,$ by continuity of $i$ and elements belonging to $C(X),$ there are equal everywhere. 
This means that 
\begin{equation}\label{ep}
    fg = FG + H,
\end{equation}
 because $i$ is injective. Similarly, for $x\in Q$
\begin{equation}\label{eq:local-233}
    \begin{split}
        (i(f))(x)-(i(F))(x)  &= 
         \langle \delta_x, i(f-F)\rangle\\
        &= 
           \langle i^*(\delta_x), f-F\rangle\\
        &= 
        \lim_{n\to\mathscr{U}} \langle i^*(\delta_x), F_{n}-F\rangle\\
        &= 
         \lim_{n\to\mathscr{U}} \langle \delta_x, i(F_{n}-F)\rangle\\
        &= 
        \lim_{n\to\mathscr{U}}\big( (i(F_{n}))(x)-(i(F))(x)\big)\\
        &= 
        \lim_{n\to\mathscr{U}}\sum_{j=0}^{n}\big( (i(F_{j+1}))(x)-(i(F_j(x)\big)\\
    \end{split}
\end{equation}
but from \eqref{eq:local-244} we know that
\[
    \sum_{n=0}^\infty \big\| i(F_{n+1})-i(F_n)\big\|_{\infty} \leqslant  C \cdot \sum_{n=0}^{\infty} \|F_{n+1}-F_n\|_{A}, 
\]
so for any $x\in Q$
\[
 (i(f))(x)-(i(F))(x) = \sum_{n=0}^\infty \big( (i(F_{n+1}))(x)-(i(F_n))(x)\big),
\]
hence, again by density of $Q$ in $X,$ continuity of $i$ and elements belonging to $C(X),$ we have this equality everywhere. Moreover, since $i$ is injective, we obtain
\[
    f-F = \sum_{n=0}^\infty \big( F_{n+1}-F_n\big),
\]
so from \eqref{eq:local-244} we have 
\begin{equation}\label{eq:local-24}
    \begin{split}
        \|f-F\|_{A}
        &\leqslant 
        \sum_{n=0}^\infty \|F_{n+1}-F_n\|_{A}
         < \e.
    \end{split}
\end{equation}
We show that $\|g-G\|_{A}<\e$ in the same way.
\end{itemize}

In each of the above cases, we have obtained the appropriate functions $f$ and $g$, which, to simplify the notation, have been marked with the same symbols. So, for every $H \in A$ satisfying \eqref{eq:local-3}, there 
exist $f$ and $g$ in $A$ such that 
\[
\|f-F\|_{A}<\e, \quad \|g-G\|_{A}<\e
\] 
(see respectively \eqref{in} or \eqref{eq:local-24}) and $FG + H = fg$ (see respectively \eqref{eq:local-223} or \eqref{ep}). 
This means that
\[
    B_{A}(F\cdot G,\delta)
    \subset 
    B_{A}(F,\e)\cdot B_{A}(G,\e)
\]
with $\delta:=\e\cdot \frac{\gamma}{C K^3 T^2}$. Hence, the multiplication in $A$ is locally open at the pair $(F,G)\in A^2$.

Suppose now that $i$ has dense range in $C(X)$. By inverse-closedness of $A$, $A$ has topological stable rank 1 if and only if $C(X)$ has topological stable rank 1. Consequently, if $C(X)$ fails to have dense invertibles (which happens exactly when $\dim X > 1$), then $A$ does not have open multiplication. 
\end{proof}

Applying Theorem~\ref{lem:diff} we obtain the following conclusion.

\begin{corollary}
        Suppose that $A$ is a unital Banach *-algebra such that there exists an injective *-homomorphisim $i\colon A\to C(X)$ for some compact space $X$ such that $A$ is a differential subalgebra of $C(X)$. Let us consider either case:
    \begin{itemize}
        \item $A = C(X)$, 
        \item $A = E^*$ is a dual Banach algebra that shares with $X$ densely many points.
    \end{itemize}
    Then multiplication in $A$ is open at all pairs of jointly non-degenerate elements. 
\end{corollary}
\begin{corollary} 
    Let $A$ be a (complex) reflexive Banach space with a $K$-unconditional basis $(e_\gamma)_{\gamma\in \Gamma}$ $(K\geqslant 1)$. Then $A$ is naturally a Banach *- algebra when endowed with multiplication 
    \[
        a \cdot b = \sum_{\gamma\in \Gamma} a_\gamma b_\gamma e_\gamma \quad (a = \sum_{\gamma\in \Gamma} a_\gamma e_\gamma, b = \sum_{\gamma\in \Gamma} b_\gamma e_\gamma\in A).
    \]
    and coordinate-wise complex conjugation. Let $A^\#$ denote the unitisation of $A$. Then $A^\#$ has open multiplication. 
\end{corollary}
\begin{proof} It is clear any pair of elements of $A^\#$ is approximable by jointly non-degenerate products (see Definition \ref{df23}). Since the basis $(e_\gamma)_{\gamma\in \Gamma}$ is $K$-unconditional, we have
    \begin{equation*}
    \begin{split}
        \|ab\|_A &=
        \bigg\| \sum_{\gamma\in \Gamma}  a_\gamma b_\gamma e_\gamma \bigg\|_A \\
        &\leqslant K \bigg\|\sum_{\gamma\in \Gamma}  a_\gamma \cdot \|b\|_{\ell_\infty(\Gamma)} \cdot e_\gamma \bigg\|_A \\
        &=K\|a\|_A \|b\|_{\ell_\infty(\Gamma)}\\
        &\leqslant K(\|a\|_A \|b\|_{\ell_\infty(\Gamma)}+\|a\|_{\ell_\infty(\Gamma)}\|b\|_A).
    \end{split}
    \end{equation*}
This means that $A^\#$ is a differential subalgebra of $c(\Gamma)$, the unitisation of the algebra of functions that vanish at infinity on $\Gamma$. Since the formal inclusion from $A^\#$ to $c(\Gamma)$ has dense range, the conclusion follows.
\end{proof}

We now turn our attention to Theorem~\ref{thm:B}.
\begin{proof}[Proof of Theorem~\ref{thm:B}]
    By Lemma~\ref{lem:idempotent}, there exists an ultrafilter $\mathscr{U}$ such that $\mathbb Z^{(\mathbb R)}$ embeds into $G^{\mathscr{U}}$. As $\mathbb Z^{(\mathbb R)}$ is a free Abelian group, it admits a surjective homomorphism $\varphi$ onto $\mathbb{Q}^{(\mathbb N)}$. Since $\mathbb{Q}^{(\mathbb N)}$ is divisible, it is an injective object in the category of Abelian groups, so $\varphi$ extends to a~homomorphism $\overline{\varphi}\colon G^{\mathscr{U}}\to \mathbb{Q}^{(\mathbb N)}$. In particular, the infinite-dimensional space $\widehat{\mathbb{Q}^{(\mathbb N)}}\cong \mathbb T^{\mathbb N}$ embeds topologically into $\widehat{G^{\mathscr{U}}}$. 
    
    Consequently, $\dim \widehat{G^{\mathscr{U}}} = \infty > 1$. By \cite[Corollary 4.10]{DraKa}, multiplication in $\ell_1(G^{\mathscr{U}})$ is not open. However, $\ell_1(G^{\mathscr{U}})$ is a quotient of the Banach-algebra ultrapower $(\ell_1(G))^{\mathscr{U}}$ (\cite[Section 2.3.2]{DraKa}), so by \cite[Corollary 3.3]{DraKa}, convolution in $\ell_1(G)$ is not uniformly open.
\end{proof}

\section{Proof of Theorem~\ref{thm:c}}
The present section is devoted to the proof of Theorem~\ref{thm:c}. We start by proving a special case of $X = [0,1]$; the argument is a slightly improved version of a proof due to Behrends. We are indebted for his permission to include it here.

\begin{theorem}\label{th4.1}    
    The (complex) algebra $C[0,1]$ has uniformly open multiplication.
\end{theorem}

In order to prove Theorem~\ref{th4.1} we require further auxiliary results.

Anywhere below $\Delta$ will denote a set  of all $(\alpha, \beta, \gamma) \in \mathbb{C}^{3}$ such that $|\gamma|=1$ and the polynomial $\gamma z^{2}+\beta z+\alpha$ has two roots of different absolute value. In particular, in this situation, there is a uniquely determined root, so  we can introduce the following definition

\begin{definition}
    We denote by $Z\colon \Delta \rightarrow \mathbb{C}$ the map that assigns to $(\alpha, \beta, \gamma)$ the root of the quadratic polynomial $\gamma z^{2}+\beta z+\alpha$ with the smaller absolute value.
\end{definition}
\begin{remark}\label{rm2}
    The root function is locally analytic, so the function $Z$ is continuous.
\end{remark}

Let us now fix a non-degenerate interval $\left[a_{0}, b_{0}\right]$.

We denote by $\overline{I}$ (respectively $I^{o}$) the closure (respectively, the interior) of an interval.

\begin{lemma}\label{l4.8}
   For any function $h\in C[a_{0}, b_{0}]$ and arbitrary $\eta_{2}>\eta_{1}>0$ there are pairwise disjoint closed subintervals $J_{1}, \ldots, J_{k}$ of $\left[a_{0}, b_{0}\right]$ such that
    \[
    \left\{t\in [a_{0}, b_{0}] \colon |h(t)| \leqslant \eta_{1}\right\} \subset \bigcup_{j=1}^k J_j \subset\left\{t \in [a_{0}, b_{0}]\colon |h(t)|<\eta_{2}\right\}.
    \]
\end{lemma}

\begin{proof}
    Consider the sets $K:=\left\{t \colon |h(t)| \leqslant \eta_{1}\right\}$ and $O:=\left\{t \colon |h(t)|<\eta_{2}\right\}$. Since $K \subset O$, for any $t \in K$ we may find an open subinterval $I_{t}$ so that $t \in I_{t} \subset \overline{I_{t}} \subset O.$ As $K$ is compact, it is possible to cover $K$ with finitely many such intervals, whose closures are the sought sets $J_i$ (it might be necessary to pass to unions if they are not disjoint).
\end{proof}

\begin{lemma}
    Let $h_{1}, h_{2}\in C\left[a_{0}, b_{0}\right]$. Suppose that $h_1$ and $h_2$ are jointly $\eta^2$-non-degenerate for some $\eta >0 $. Then there are continuous $\beta_{1}, \beta_{2}\colon \left[a_{0}, b_{0}\right] \rightarrow \mathbb{T}$ such that 
    \[
        \left|h_{1}(t) \beta_{1}(t)+h_{2}(t) \beta_{2}(t)\right| \geqslant \eta\quad (t \in\left[a_{0}, b_{0}\right]).
    \]
\end{lemma}

\begin{proof}
    By compactness, we may find $\eta_{0}>0$ such that that $h_1$ and $h_2$ are jointly $(\eta^{2}+\eta_{0}^{2})$-non-degenerate; we also will assume that $2 \eta_{0}^{2}<\eta^{2}$. Next we choose, with a $\tau \in[0,1]$ that will be fixed later, pairwise disjoint closed intervals $J_{1}, \ldots, J_{k}$ and pairwise disjoint closed intervals $J_{k+1}, \ldots, J_{l}$ such that
    \[ 
    \left\{t\colon |h_{1}(t)| \leqslant \tau\cdot \frac{\eta_{0}}{2}\right\} \subset \bigcup_{j=1}^{k} J_{j} \subset\left\{t\colon |h_{1}(t)|<\tau\cdot \eta_{0}\right\}
    \]
and
    \[ 
    \left\{t\colon |h_{2}(t)| \leqslant \tau\cdot \frac{\eta_{0}}{2}\right\} \subset \bigcup_{j=k+1}^{l} J_{j} \subset\left\{t\colon |h_{2}(t)|<\tau \cdot \eta_{0}\right\}
    \]
    (see Lemma \ref{l4.8}). As a consequence of $2 \eta_{0}^{2}<\eta^{2}$ no $J_{j}$ with $j \leqslant k$ intersects  $J_{j^{\prime}}$ with $j^{\prime}>k\colon$ the family $\left(J_{j}\right)_{j=1}^l$ comprises disjoint intervals.

    We now define $\beta_{1}$ and $\beta_{2}$. The function $\beta_{1}$ is the function constantly equal to one, and $\beta_{2}$ is constructed as follows. On $\left[a_{0}, b_{0}\right] \backslash \bigcup_{j=1}^{l} J_{j}^{o}$ we put
    \[
        \beta_{2}(t):=i \frac{h_{1}(t) \overline{h_{2}(t)}}{\left|h_{1}(t) \overline{h_{2}(t)}\right|}.
    \]
     The values are in $\mathbb{T}$ so that, by the Tietze extension theorem, we may find a $\mathbb{T}$-valued continuous extension to all of $\left[a_{0}, b_{0}\right]$ that will be also denoted by $\beta_{2}$.\smallskip

    We \emph{claim} that $\beta_{1}$ and $\beta_{2}$ have the desired properties. By construction both functions are continuous and they satisfy $\left|\beta_{1}(t)\right|=\left|\beta_{2}(t)\right|=1$ for all $t$. For $t \in$ $\left[a_{0}, b_{0}\right] \backslash \bigcup_{j=1}^{l} J_{j}^{o}$ a simple calculation gives us
    \[ 
    \left|h_{1}(t) \beta_{1}(t)+h_{2}(t) \beta_{2}(t)\right|^{2}=\left|h_{1}(t)\right|^{2}+\left|h_{2}(t)\right|^{2}\geqslant\eta^{2} + \eta_0^{2} .
    \]

    Now fix $t$ in one of the $J_{j}$ with $j \leqslant k$. Then $\left|h_{1}(t)\right|^{2}<\tau^{2} \eta_{0}^{2}$ so that $\left|h_{2}(t)\right| \geqslant \sqrt{\eta^{2}+\left(1-\tau^{2}\right) \eta_{0}^{2}}$. We may then continue our estimation as follows:
    \[ 
        \begin{split}
        \left|h_{1}(t) \beta_{1}(t)+h_{2}(t) \beta_{2}(t)\right| 
        & \geqslant \left|h_{2}(t)\right|-\left|h_{1}(t)\right| \\
        & \geqslant \sqrt{\eta^{2}+\left(1-\tau^{2}\right) \eta_{0}^{2}}-\tau \eta_{0} \\
        & \geqslant \eta.
        \end{split}
    \]
 The last inequality holds if we choose $\tau$ small enough. The argument for $\bigcup_{j=k+1}^{l} J_{j}$ is analogous.
\end{proof}

\begin{lemma}\label{l4.10}
    Let $h_{1}, h_{2}\in C\left[a_{0}, b_{0}\right]$. Suppose that $h_1, h_2$ are jointly non-degenerate. Then for every $\varepsilon>0$ there is a positive $\delta$ such that for every $d\in C\left[a_{0}, b_{0}\right]$ satisfying $\|d\| \leqslant \delta$ there are $z_{1}, z_{2}\in C\left[a_{0}, b_{0}\right]$ such that 
    \begin{itemize}
        \item $\left|z_{1}(t)\right|,\left|z_{2}(t)\right| \leqslant \varepsilon$ and \item $h_{1}(t) z_{1}(t)+h_{2}(t) z_{2}(t)+z_{1}(t) z_{2}(t)=d(t)$ $(t \in\left[a_{0}, b_{0}\right])$.
    \end{itemize}
\end{lemma}

\begin{proof}
    Choose $\beta_{1}, \beta_{2}$ as in the preceding lemma and put $f:=h_{1} \beta_{1}+h_{2} \beta_{2}$ and $g:=$ $\beta_{1} \beta_{2}$. Then $|f|\geqslant \eta$ for some $\eta > 0$  and $|g|=1$. We conclude the proof by showing that for every $\varepsilon>0$ there is $\delta>0$ such that if $d\in C[a_{0}, b_{0}]$ and $\|d\| \leqslant \delta$ there is $\phi \in C[a_{0}, b_{0}]$ with $\|\phi\| \leqslant \varepsilon$ and 
    \[
        f(t) \phi(t)+g(t) \phi^{2}(t)=d(t) \quad(t\in \left[a_{0}, b_{0}\right]).
    \]
    Indeed, then we can set $z_{1}:=\beta_{1} \phi$ and $z_{2}:=\beta_{2} \phi$.

    Fix $(\alpha, \beta, \gamma) \in \mathbb{C}^{3}$ satisfying $|\beta| \geqslant \eta,$ $|\gamma|=1$ and arbitrary, strictly positive $\eta, \varepsilon.$ 
    It is enough to find $\delta>0$ such that if  $|\alpha| \leqslant \delta$ then $(\alpha, \beta, \gamma) \in \Delta$ and $|Z(\alpha, \beta, \gamma)| \leqslant \varepsilon$. Indeed, by Remark \ref{rm2}, this allows us to define function $\phi$ as  $\phi(t) :=  Z(-d(t), f(t), g(t))$ for $t\in [a_0,b_0].$
    
    Denote by $z_{1}, z_{2}$ the roots of the polynomial $\gamma z^{2}+\beta z+\alpha$. By Vieta's formulae
    \[
        \gamma\left(z_{1}+z_{2}\right)=-\beta,
    \] 
    hence either $\left|z_{1}\right| \geqslant \eta / 2$ or $\left|z_{2}\right| \geqslant \eta / 2$. Without loss of generality we may assume that $\left|z_{1}\right| \geqslant \eta / 2.$

    Again, by Vieta's formulae, 
    \[
        \gamma z_{1} z_{2}=\alpha,
    \] 
    so that $z_{2}=\alpha/(\gamma z_1)$, hence $\left|z_{2}\right| \leq$ $2|\alpha| / \eta$. Thus, it suffices to choose $|\alpha| \leqslant \delta$ where $\delta>0$ satisfies $2 \delta / \eta \leqslant \varepsilon$ and $2 \delta / \eta<\eta / 2$. Then $\left|z_{2}\right|<\left|z_{1}\right|$ and $\left|z_{2}\right| \leqslant \varepsilon,$ so conclusion follows by the definition of $Z.$    
\end{proof}

\begin{lemma}\label{l4.12}
    Let $\varepsilon > 0$ and $\psi\in C[a, b]$. Suppose that $\inf_{t \in[a, b]} |\psi(t)| \leqslant \varepsilon^{2}$. If there are $Z_{a}, W_{a}, Z_{b}, W_{b} \in \mathbb{C}$ such that $Z_{a} W_{a}=\psi(a), Z_{b} W_{b}=\psi(b)$ and $\left|Z_{a}\right|,\left|W_{a}\right|,\left|Z_{b}\right|,\left|W_{b}\right| \leqslant \varepsilon$, then there are  $Z_{1}, Z_{2}\in C[a, b]$ with the following properties:

    \begin{itemize}
    \item $Z_{1}(a)=Z_{a}, Z_{2}(a)=W_{a}, Z_{1}(b)=Z_{b}, Z_{2}(b)=W_{b},$ 
    \item $\left|Z_{1}(t)\right|,\left|Z_{2}(t)\right| \leqslant \varepsilon \text{ and } Z_{1}(t) Z_{2}(t)=\psi(t) \text{ for all } t.$
    \end{itemize}
\end{lemma}

\begin{proof}
    Let $b^{\prime} \in (a,b)$ and take $\hat{Z}$ with $\hat{Z}^{2}=\psi\left(b^{\prime}\right)$. We will define the sought functions $Z_{1}$ and $Z_{2}$ on the interval $\left[a, b^{\prime}\right]$. In the case of $\left[b^{\prime}, b\right]$, we simply repeat the procedure and glue $Z_{1}, Z_{2}$ together as defined on these subintervals.

    Without loss of generality we may suppose that $\left|Z_{a}\right| \geqslant \left|W_{a}\right|$ so that $\left|Z_{a}\right| \geqslant \sqrt{|\psi(a)|}$. Choose $Z_{1}\in [a, b']$ with 
    \begin{itemize}
        \item $Z_{1}(a)=Z_{a}$,
        \item $Z_{1}(b')=\hat{Z}$, and 
        \item $\varepsilon > \left|Z_{1}(t)\right| \geqslant \sqrt{|\psi(t)|}$ for all $t$.
    \end{itemize}
    Let us observe that $|\hat{Z}| \geqslant \sqrt{|\psi(b)|}$\footnote{Here it is important that we work in $\mathbb{C}$ and not in $\mathbb{R}$.}. We may then define
        \[   
            Z_{2}(t):= 
                 \begin{cases}
                   0 &\quad\text{if } Z_{1}(t)=0, \\
                   \psi(t) / Z_{1}(t)&\quad\text{otherwise} .
                 \end{cases}
        \]
    Then $Z_{1}$ and $Z_{2}$ will have the claimed properties. Indeed, the continuity of $Z_{2}$ at points $t_{0}$ with $Z_{1}\left(t_{0}\right)=0$ is proved as follows. If $Z_{1}\left(t_{0}\right)=0$ then $\psi\left(t_{0}\right)=0$. Thus, by continuity of $\psi,$ if $t_{n} \rightarrow t_{0}$, then $\sqrt{\left|\psi\left(t_{n}\right)\right|} \rightarrow 0.$ Hence $\left|Z_{2}\left(t_{n}\right)\right|=\left|\psi\left(t_{n}\right) / Z_{1}\left(t_{n}\right)\right| \leqslant \sqrt{\left|\psi\left(t_{n}\right)\right|}$ will tend to zero as well.
    
\end{proof}

\begin{proof}[Proof of Theorem \ref{th4.1}]
    Let $\varepsilon_{0}>0$. We have to find $\delta_{0}>0$ with the following property: whenever $d\colon[0,1] \rightarrow \mathbb{C}$ is a prescribed continuous function with $\|d\| \leqslant \delta_{0}$ it is possible to find functions $d_{1}, d_{2} \in C[0,1]$ with $\left\|d_{1}\right\|,\left\|d_{2}\right\| \leqslant \varepsilon_{0}$ and $\left(f+d_{1}\right)\left(g+d_{2}\right)=f g+d$ (\emph{i.e.}, $\left.f d_{2}+g d_{1}+d_{1} d_{2}=d\right)$ for any $f, g \in C[0,1]$. Fix $f, g \in C[0,1].$\smallskip

    The idea is to determine such $d_{1}, d_{2}$ by using Lemma \ref{l4.10} (Lemma \ref{l4.12}, respectively) on the subintervals where the functions $f$ and $g$ are jointly non-degenerate (respectively, jointly degenerate) and to glue the pieces together.\smallskip

    With an $\varepsilon_{1}>0$ that will be fixed later we apply Lemma \ref{l4.8} with $h:=|f|^{2}+|g|^{2}$ and $\eta_{1}:=\varepsilon_{1}^{2}, \eta_{2}:=4 \varepsilon_{1}^{2}$. Write the intervals $J_{j}$ $(j=1, \ldots, k)$ as $J_{j}=\left[a_{j}, b_{j}\right]$, where, without loss of generality, $0 \leqslant a_{1}<b_{1}<a_{2}<b_{2}<\cdots<a_{k}<b_{k}$. Note that $h(t) \leqslant 4 \varepsilon_{1}^{2}$ on each $\left[a_{j}, b_{j}\right]$ and $h(t)>\varepsilon_{1}^{2}$ on the intervals $\left[b_{j}, a_{j+1}\right]$.\smallskip

    Let us consider the intervals $\left[b_{j}, a_{j+1}\right]$ and apply Lemma \ref{l4.10} with $\left[a_{0}, b_{0}\right]:=$ $\left[b_{j}, a_{j+1}\right]$, $\eta:=\varepsilon_{1}$, and $\varepsilon:=\varepsilon_{1}$. Choose $\delta$ as in the lemma; without loss of generality we may assume that $\delta \leqslant \varepsilon_{1}^{2}$. We consider any $d \in C[0,1]$ with $\|d\| \leqslant \delta$. Lemma \ref{l4.10} provides continuous $z_{1}, z_{2}\colon \left[b_{j}, a_{j+1}\right] \rightarrow \mathbb{C}$ with $f(t) z_{1}(t)+g(t) z_{2}(t) + z_1z_2 = d(t)$ and $\left|z_{1}(t)\right|,\left|z_{2}(t)\right| \leqslant \varepsilon_{1}$ for $t \in\left[b_{j}, a_{j+1}\right]$. We define $d_{1}$ ($d_{2}$, respectively ) on $\left[b_{j}, a_{j+1}\right]$ by $z_{2}$ ($z_{1}$, respectively). Then $\left(f+d_{1}\right)\left(g+d_{2}\right)=f g+d$ on these subintervals. (It should be noted here that the $\delta$ in Lemma \ref{l4.10} does only depend on $\eta$ and $\varepsilon$ but not on $a_{0}, b_{0}$.)\smallskip

    Now $d_{1}, d_{2}$ are suitably defined on the union of the $\left[b_{j}, a_{j+1}\right]$. The gaps will be filled with the help of Lemma \ref{l4.12}. Consider any $\left[a_{j}, b_{j}\right]$. For a $t$ in such an interval\\
    we know that $|f(t)|,|g(t)| \leqslant 2 \varepsilon_{1}$ so that $|f(t) g(t)| \leqslant 4 \varepsilon_{1}^{2}$.

    It follows that $\psi\colon \left[a_{j}, b_{j}\right] \rightarrow \mathbb{C}, t \mapsto f(t) g(t)+d(t)$ satisfies $|\psi(t)| \leqslant 5 \varepsilon_{1}^{2} \leqslant  (5 \varepsilon_{1})^{2}$. We apply Lemma \ref{l4.12} with this function $\psi$ and
    \[ Z_{a}:=\left(f+d_{1}\right)\left(a_{j}\right), W_{a}:=\left(g+d_{2}\right)\left(a_{j}\right), Z_{b}:=\left(f+d_{1}\right)\left(b_{j}\right), W_{b}:=\left(g+d_{2}\right)\left(b_{j}\right)
    \]
    and $\varepsilon:=5 \varepsilon_{1}$. It remains to use the functions $Z_{1}, Z_{2}$ found by the lemma to define $d_{1}, d_{2}$ on $\left[a_{j}, b_{j}\right]$. Here $Z_{1}$ (respectively $Z_{2}$) plays the r\^{o}le of $f+d_{1}$ ($g+d_{2}$) so that we may set $d_{1}(t):=Z_{1}(t)-f(t)$ and $d_{2}(t):=Z_{2}(t)-g(t)$ for $t \in\left[a_{j}, b_{j}\right]$. At the endpoints this assignment is compatible with the previous one: at $a_{j}$, e.g., $d_{1}$ was already defined, but as a consequence of $Z_{1}(a)=Z_{a}=f\left(a_{j}\right)+d\left(a_{j}\right)$ the new definition of $d_{1}\left(a_{j}\right)$ as $\left(Z_{1}-f\right)\left(a_{j}\right)$ leads to the same value.

    We observe that $\left|d_{j}(t)\right| \leqslant (2+5)\varepsilon_{1} = 7 \varepsilon_{1}$ for $j=1,2$ so that we may summarise the above calculations as follows: if one starts with $\varepsilon_{1}:=\varepsilon_{0} / 7$, then $\delta_{0}:=\delta$ with the $\delta$ that we have just found has the desired properties.\smallskip

    It should be noted that our proof is not yet complete since when considering the $\left[a_{j}, b_{j}\right]$, our argument used the fact that the functions $d_{1}, d_{2}$ were already defined at $a_{j}$ and $b_{j}$, so we are to consider the cases $a_{1}=0$ or $b_{k}=1$. If, \emph{e.g.}, $a_{1}=0$ we choose any $Z_{a}, W_{a}$ with $\left|Z_{a}\right|,\left|W_{a}\right| \leqslant \varepsilon$ and $Z_{a} W_{a}=\psi(a)$; we proceed similarly for $b_k=1$.
\end{proof}

\subsection{Uniform openness of multiplication in \texorpdfstring{$C(X)$}{TEXT}}
The next result is crucial for establishing the only non-trivial implication in Theorem~\ref{thm:c}.
\begin{theorem}
    Let $X$ be a compact space of covering dimension at most $1$. Then multiplication in $C(X)$ is uniformly open.
\end{theorem}
\begin{proof}
    \emph{Case 1:} $X$ is a topological realisation of a graph in the complex plane. \smallskip
    
    We \emph{claim} that $C(X)$ has uniformly open multiplication and $\delta(\varepsilon)$ does not depend on $X$ in the class of such graphs, that is, multiplications in $C(X)$ are equi-uniformly open for all graphs $X$.\smallskip
    
    For this, let us consider a partition of $X$ into finitely many intervals, $\bigcup_{j=1}^{k} [a_j, b_j]$. We define a finer partition of this graph into intervals as follows. If the intervals $[a_j, b_j]$ and $[a_i, b_i]$ intersect at $c$ for some $j,i \in \{1,\ldots k\}$, then $c$ must be the endpoint of the intervals, \emph{i.e}., we replace the interval $[a_j, b_j]$ by sub-intervals $[a_j, c]$ and $[c, b_j]$ whenever $c\in (a_j,b_j)$ (analogously for the interval $[a_i, b_i]$). For each interval in the new partition $P=\bigcup_{j=1}^{K} [a_j, b_j]$ we apply a procedure analogous to the one in the proof of Theorem \ref{th4.1}. \smallskip
    
    More precisely, denote for any function $F\colon P\rightarrow \mathbb{C}$ its restriction to the interval $[a_j,b_j]$ by $F^j.$ Then for $\varepsilon_{0}>0$ find a positive $\delta_{0}$ with the following property: whenever $d\in C(P)$, $\|d\| \leqslant \delta_{0}$ for every restriction $d^j\in C[a_j, b_j]$ ($j \in \{1,\ldots,K\}$) we may find $d_{1}^j, d_{2}^j \in C[a_j, b_j]$ with $\|d_{1}^j\|,\|d_{2}^j\| \leqslant \varepsilon_{0}$ and $(f^j+d_{1}^j)(g^j+d_{2}^j)=f^j g^j+d^j$ for any $f,g\colon P \rightarrow \mathbb{C}.$ 
    
    We glue the functions $d_1^j$ for all $j \in \{1,\ldots,K\}$ to obtain a function $d_1\colon P\rightarrow \mathbb{C}$ (analogously, we get $d_2$). Note that due to the choice of partition $P$, these functions are uniquely defined at the endpoints of the intervals, because at the intersection points of the intervals, we always take the same value of the function. It should be noted also (again) that the $\delta$ in Lemma \ref{l4.10} does only depend on $\eta$ and $\varepsilon$ and not on $a_{0}, b_{0}$.\medskip

    \noindent \emph{Case 2}: $X$ is a compact metric space of covering dimension at most 1. \smallskip 
    
    It is known that for a~zero-dimensional (not necessarily metrisable) compact space $X$, $C(X)$ has uniformly open multiplication with $\delta(\varepsilon) = \varepsilon^2 / 4$ (\cite[Proposition 4.6]{DraKa}). In the light of Case 1, by taking minimum if necessary, we may suppose that $\delta(\varepsilon)$ is the same for all zero-dimensional spaces as well as all graphs in the plane. However, every one-dimensional compact metric space $X$ is the projective limit of an inverse sequence $(K_i, \pi_i^j)$ of at most one-dimensional `polyhedra' (this is a theorem of Freudenthal \cite{freudenthal1937entwicklungen}; see \cite[Theorem 1.13.2]{engelking1978dimension} for modern exposition), \emph{i.e.}, finite sets and graphs in the plane. Such an inverse sequence gives rise to a direct system $(C(K_i), h_{\pi_i^j})$, where $h_{\pi_i^j}$ is a *-homomorphic embedding of $C(K_i)$ into $C(K_j)$ ($i\leqslant j$) given by 
    \[
        h_{\pi_i^j}f = f\circ \pi_i^j\quad (f\in C(K_i)).
    \]
    As $C(X)$ is naturally *-isomorphic to the completion of the chain $(C(K_i), h_{\pi_i^j})$ (\emph{i.e.}, the C*-direct limit; see \cite[Section 1]{takeda1955inductive} for more details) in which multiplications are equi-uniformly open, by \cite[Corollary 3.6]{DraKa}, $C(X)$ has uniformly open multiplication and $\delta(\varepsilon)$ depends only on $\varepsilon$ but not the compact metric space $X$ considered.\medskip

    \noindent\emph{Case 3}: $X$ is an arbitrary compact space of covering dimension at most 1.\smallskip

    By \cite[Theorem 1]{mardevsic1960covering} every compact space $X$ is an inverse limit of a well-ordered system of metrisable compacta $X_{\alpha}$ with $\dim X_{\alpha} \leqslant \dim X.$ As proved in Claim 2, $C(X_\alpha)$ have equi-uniformly open multiplications, meaning that $\delta(\varepsilon)$ is the same for all items of the inverse system considered, so multiplication in $C(X)$ is uniformly open (\cite[Corollary 3.6]{DraKa}).
 
\end{proof}

\section{Open problems}In the light of Theorem~\ref{thm:A} let us pose the following question.
\begin{question}
    What are further examples of (dual) Banach algebras that are approximable by jointly non-degenerate elements? What about algebras of Lipschitz functions on zero-dimensional compact spaces?
\end{question}
In the case of convolution algebras on discrete groups having at most one-dimensional dual groups, we ask the following question.
\begin{question}
    Can the group algebra of a group with bounded exponent have (uniformly) open convolution?
\end{question}
More generally:
\begin{question}
    Is there an infinite group $G$ for which $\ell_1(G)$ has open convolution?
\end{question}

\subsection*{Acknowledgements} Support from SONATA 15 No.~2019/35/D/ST1/01734 is acknowledged with thanks. The second-named author was supported by GAČR grant GF20-22230L and received an incentive scholarship from the funds of the program Excellence Initiative - Research University at the Jagiellonian University in Krak\'{o}w. We are indebted to Lav Kumar Singh for careful reading of the manuscript. Finally, we are grateful to the referee for their insightful remarks.

\bibliography{literature.bib}
\bibliographystyle{plain}

\end{document}